\documentclass[12pt,oneside]{article}
\usepackage{amsmath,amssymb,amsfonts,amsthm}
\usepackage{centernot}
\usepackage{stmaryrd}
\usepackage{lscape}
\newcommand{\T}{{\cal T}}

\newcommand {\cp}{\mathfrak{X}(\pi (M))}

\def\Section#1{\vspace{30truept}\addtocounter{section}{1}\setcounter{thm}{0}
\setcounter{equation}{0}{\noindent\Large\bf
    \arabic{section}.~~#1}\par \vspace{12pt}}

\newtheorem{thm}{Theorem}[section]

\newtheorem{prop}[thm]{Proposition}
\newtheorem{defn}[thm]{Definition}

\newtheorem{rem}[thm]{Remark}

\numberwithin{equation}{section}

\begin{document}
\title{{\textbf{New Special Finsler Spaces}}}  
\author{{\bf Nabil L. Youssef\,$^{\,1}$ and A. Soleiman$^{2}$ }}
\date{}

\maketitle                     
\vspace{-1.15cm}
\begin{center}
{$^{1}$Department of Mathematics, Faculty of Science,\\ Cairo
University, Giza, Egypt\\nlyoussef@sci.cu.edu.eg, nlyoussef2003@yahoo.fr}
\end{center}
\begin{center}
{$^{2}$Department of Mathematics, Faculty of Science,\\ Benha
University, Benha, Egypt\\amr.hassan@fsci.bu.edu.eg, amrsoleiman@yahoo.com}
\end{center}

\vspace{0.2cm}
\maketitle

\vspace{0.2cm}
\hfill \emph{Dedicated to the memory of Waleed A. Elsayed}
\vspace{0.6cm}


\noindent{\bf Abstract.}
The pullback  approach to global Finsler geometry is adopted.  Some new types of special Finsler spaces are introduced and investigated, namely, Ricci, generalized Ricci, projectively recurrent and m-projectively recurrent Finsler spaces. The properties of these special Finsler spaces are studied and the relations between them are singled out.


\bigskip


\begin{center}
\large{\bf Introduction}
\end{center}
Many types of recurrence in Riemannian geometry have been studied by many authors \cite{R3,  R4, Ra1, R2, R3aa, R3a, R5, R1}.
 On the other hand, some types of recurrence in Finsler geometry have been also studied \cite{F3, F2, F1, Types}.
\par
In a recent paper \cite{Types}, we have introduced and investigated \emph{intrinsically} three classes of recurrence in Finsler geometry: simple recurrence, Ricci recurrence and concircular recurrence. Each of these classes consists of four types of recurrence.  We also investigated the interrelationships between the different types of recurrence.
\par
The present paper is a continuation of ~\cite{Types}, where we introduce and investigate some new types of special Finsler spaces, namely, Ricci, generalized Ricci,  projectively recurrent and m-projectively recurrent Finsler spaces.  Some Finsler tensors are defined and
their properties are studied.  These tensors are used to define the projectively recurrent and m-projectively recurrent Finsler spaces. The relations between the above mentioned spaces are investigated.
\newpage

\Section{Notation and Preliminaries}

In this section, we give a brief account of the basic concepts
 of the pullback approach to intrinsic Finsler geometry necessary for this work. For more
 details, we refer to \cite{r58, r86, r62, r92,  r94, r96}. We
 shall use the notations of \cite{r86}.

 In what follows, we denote by $\pi: \T M\longrightarrow M$ the subbundle of nonzero vectors
tangent to $M$, $\mathfrak{F}(TM)$ the algebra of $C^\infty$ functions on $TM$, $\cp$ the $\mathfrak{F}(TM)$-module of differentiable sections of the pullback bundle $\pi^{-1}(T M)$.
The elements of $\mathfrak{X}(\pi (M))$ will be called $\pi$-vector
fields and will be denoted by barred letters $\overline{X} $. The
tensor fields on $\pi^{-1}(TM)$ will be called $\pi$-tensor fields.
The fundamental $\pi$-vector field is the $\pi$-vector field
$\overline{\eta}$ defined by $\overline{\eta}(u)=(u,u)$ for all
$u\in \T M$.
\par
We have the following short exact sequence of vector bundles
$$0\longrightarrow
 \pi^{-1}(TM)\stackrel{\gamma}\longrightarrow T(\T M)\stackrel{\rho}\longrightarrow
\pi^{-1}(TM)\longrightarrow 0 ,\vspace{-0.1cm}$$ with the well known
definitions of  the bundle morphisms $\rho$ and $\gamma$. The vector
space $V_u (\T M)= \{ X \in T_u (\T M) : d\pi(X)=0 \}$  is the vertical space to $M$ at $u$.
\par
Let $D$ be  a linear connection on the pullback bundle $\pi^{-1}(TM)$.
  The vector space $H_u (\T M)= \{ X \in T_u
(\T M) : D_X \overline{\eta}=0 \}$ is called the horizontal space to $M$ at $u$ .
   The connection $D$ is said to be regular if
$$ T_u (\T M)=V_u (\T M)\oplus H_u (\T M) \,\,\,  \forall \, u\in \T M.$$

If $M$ is endowed with a regular connection, then the vector bundle
morphisms $ \gamma$ and $\rho |_{H(\T M)}$  are vector bundle isomorphisms.
The map  $\beta:=(\rho |_{H(\T M)})^{-1}$ is called the horizontal map of the connection $D$.
\par
 The horizontal ((h)h-) and
mixed ((h)hv-) torsion tensors of $D$, denoted by $Q $ and $ T $
respectively, are defined by \vspace{-0.2cm}
$$Q (\overline{X},\overline{Y})=\textbf{T}(\beta \overline{X}\beta \overline{Y}),
\, \,\,\, T(\overline{X},\overline{Y})=\textbf{T}(\gamma
\overline{X},\beta \overline{Y}) \quad \forall \,
\overline{X},\overline{Y}\in\mathfrak{X} (\pi (M)),\vspace{-0.2cm}$$
where $\textbf{T}$ is the (classical) torsion tensor field
associated with $D$.
\par
The horizontal (h-), mixed (hv-) and vertical (v-) curvature tensors
of $D$, denoted by $R$, $P$ and $S$
respectively, are defined by
$$R(\overline{X},\overline{Y})\overline{Z}=\textbf{K}(\beta
\overline{X}\beta \overline{Y})\overline{Z},\quad
 {P}(\overline{X},\overline{Y})\overline{Z}=\textbf{K}(\beta
\overline{X},\gamma \overline{Y})\overline{Z},\quad
 {S}(\overline{X},\overline{Y})\overline{Z}=\textbf{K}(\gamma
\overline{X},\gamma \overline{Y})\overline{Z}, $$
 where $\textbf{K}$
is the (classical) curvature tensor field associated with $D$.
\par
The contracted curvature tensors of $D$, denoted by $\widehat{{R}}$, $\widehat{ {P}}$ and $\widehat{ {S}}$ (known
also as the (v)h-, (v)hv- and (v)v-torsion tensors respectively), are defined by
$$\widehat{ {R}}(\overline{X},\overline{Y})={ {R}}(\overline{X},\overline{Y})\overline{\eta},\quad
\widehat{ {P}}(\overline{X},\overline{Y})={
{P}}(\overline{X},\overline{Y})\overline{\eta},\quad \widehat{
{S}}(\overline{X},\overline{Y})={
{S}}(\overline{X},\overline{Y})\overline{\eta}.$$
\begin{thm} {\em\cite{r94}} \label{th.1} Let $(M,L)$ be a Finsler
manifold and  $g$ the Finsler metric defined by $L$. There exists a
unique regular connection $\nabla$ on $\pi^{-1}(TM)$, called Cartan connection, such
that\vspace{-0.2cm}
\begin{description}
  \item[(a)]  $\nabla$ is  metric\,{\em:} $\nabla g=0$,

  \item[(b)] The (h)h-torsion of $\nabla$ vanishes\,{\em:} $Q=0
  $,
  \item[(c)] The (h)hv-torsion $T$ of $\nabla$\, satisfies\,\emph{:}
   $g(T(\overline{X},\overline{Y}), \overline{Z})=g(T(\overline{X},\overline{Z}),\overline{Y})$.
\end{description}
\end{thm}


\Section{Ricci (generalized Ricci) Finsler spece}

In this section, we introduce and study  some new special Finsler spaces, called Ricci and generalized Ricci Finsler spaces.
Some classes of generalized Ricci Finsler spaces are distinguished. These new spaces have been defined in Riemannian geometry \cite{R3,  R4, Ra1, R2, R3aa, R3a, R5, R1}.  We extend them to the Finslerian  case. The only linear connection we deal with in the sequel is the Cartan connection $\nabla$.
\par
For an $n$-dimensional  Finsler manifold  $(M,L)$,  we set the following notations:
\vspace{-6pt}
\begin{eqnarray*}
\stackrel{h}\nabla&:&\text{the $h$-covariant derivatives associated
with Cartan connection},\\
\text{Ric} &:& \text{the  horizontal  Ricci  tensor of Cartan connection},\\
\text{Ric}_{o} &:& \text{the  horizontal  Ricci tensor of of type (1,1) defined by }\\
 {\qquad\qquad\qquad}&& g(\text{Ric}_{o} \overline{X},\overline{Y})=\text{Ric}(\overline{X},\overline{Y})   ,\\
 r &:& \text{the horizontal scalar curvature of Cartan connection},\\
C &:=& {R}-\frac{r}{n(n-1)}\, {G}: \text{ the concircular curvature tensor};  \\
{\qquad\qquad\qquad}&& G(\overline{X},\overline{Y})\overline{Z} := g(\overline{X},\overline{Z})
\overline{Y}-g(\overline{Y},\overline{Z})\overline{X}.
\end{eqnarray*}

\begin{defn}\label{hor.} A Finsler manifold is said to be horizontally integrable if its
horizonal distribution is completely integrable or, equivalently, if $\widehat{R}=0$.
\end{defn}

\begin{defn}\label{def.1a} Let $(M,L)$ be a Finsler manifold  of dimension $n\geq3$ with non-zero Ricci tensor
 $\emph{{Ric}}$. Then, $(M,L)$ is said to be:
  \begin{description}
      \item[(a)] Ricci  Finsler manifold  if \,\,$\emph{\text{Ric}}_{o}^{2}:= \emph{\text{Ric}}_{o} \circ \emph{\text{Ric}}_{o}=\frac{r}{n-1}\,\emph{\text{Ric}}_{o}$,
   \item[(b)] generalized Ricci  Finsler manifold  if \,\,$\emph{\text{Ric}}_{o}^{2}=\alpha\,\emph{\text{Ric}}_{o}$,
    \end{description}
    where $\alpha$ is a non-zero scalar function on $TM$ called the associated scalar.
 \end{defn}

The following result gives some important properties of generalized Ricci  Finsler manifolds.

\begin{thm} \label{thm.2} Let $(M,L)$ be a generalized Ricci  Finsler manifold of dimension $n\geq3$ with associated scalar $\alpha$.
The following assertions hold:
  \begin{description}
      \item[(a)] If the Ricci  tensor is symmetric, then the scalar curvature $r$ can not  vanish.
      \item[(b)] The Ricci  tensor in the direction $\emph{\text{Ric}}_{o}(\overline{W})$; $\overline{W}$ being a non zero $\pi$-vector field,
       is the associated scalar $\alpha$.
      \item[(c)] The  tensor $\emph{\text{Ric}}_{o}$ has two eigenvalues $0$ and  $\alpha$.
      \item[(d)] If $(M,L)$ is horizontally integrable  Ricci recurrent, then the associated scalar  $\alpha=\frac{r}{2}$.
 \end{description}
\end{thm}
\begin{proof}~ \par
\noindent $\textbf{(a)}$   Let $(M,L)$ be a generalized Ricci  Finsler manifold with associated scalar $\alpha$ and $g$  the associated Finsler metric. Then, by Definition \ref{def.1a}
\begin{equation}\label{eq.1}
  \emph{\text{Ric}}(\emph{\text{Ric}}_{o}\overline{X},\overline{Y})=\alpha \emph{\text{Ric}}(\overline{X},\overline{Y}).
\end{equation}
Setting $\overline{X}=\overline{Y}=\overline{E}_{i}$, where $\{\overline{E}_{i}; i=1, ...,n\}$ is an orthonormal basis.
 Hence,
 \begin{equation*}\label{eq.2}
   \sum_{i}\emph{\text{Ric}}(\emph{\text{Ric}}_{o}\overline{E}_{i},\overline{E}_{i})=\alpha r.
 \end{equation*}
 We show that $r\neq0$. Assuming the contrary, then
 \begin{equation*}
   \sum_{i}\emph{\text{Ric}}(\emph{\text{Ric}}_{o}\overline{E}_{i},\overline{E}_{i})=0.
 \end{equation*}
 As the Ricci tensor ${\textmd{Ric}}$ is symmetric (since $(M,L)$ is horizontally integrable \cite{Types}) and $g$ is positive definite, the above relation yields $\emph{\text{Ric}}_{o}=0$, which is a contradiction.

\vspace{5pt}

\noindent $\textbf{(b)}$ Setting $\overline{X}=\overline{W}\neq0$ and $Y=\emph{\text{Ric}}_{o}\overline{W}$ in (\ref{eq.1}), we get
$$\alpha=\frac{\emph{\text{Ric}}(\emph{\text{Ric}}_{o}\overline{W},\emph{\text{Ric}}_{o}\overline{W})}
{g(\emph{\text{Ric}}_{o}\overline{W},\emph{\text{Ric}}_{o}\overline{W})},$$
which means that $\alpha$ is the Ricci  tensor in the direction $\emph{\text{Ric}}_{o}(\overline{W})$.

\vspace{5pt}

\noindent $\textbf{(c)}$  Let $\overline{V}$ be an eigenvector associated with the eigenvalue $\lambda$
 of $\emph{\text{Ric}}_{o}$, then
$$\emph{\text{Ric}}_{o} \overline{V}=\lambda \overline{V}.$$
From which, noting that $(M,L)$ is generalized Ricci with associated scalar $\alpha$, we have
$$(\lambda^{2}-\alpha\lambda)\overline{V}=0.$$
Consequently, $\lambda=0$ or $\lambda=\alpha$.

\vspace{5pt}

\noindent $\textbf{(d)}$ As $(M,L)$ is Ricci recurrent with scalar form $A$, then
\begin{eqnarray}
      ({\nabla}_{\beta\overline{X}} \emph{\text{Ric}})(\overline{Y},\overline{Z})&=&A(\overline{X})\emph{\text{Ric}}(\overline{Y},\overline{Z}) \label{eq.3}
\end{eqnarray}
and since $(M,L)$ is horizontally integrable, then, we have \cite{Types}
\footnote{$\mathfrak{S}_{\overline{X},\overline{Y},\overline{Z}}$
denotes cyclic sum over ${\overline{X},\overline{Y},\overline{Z}}$.}
\begin{eqnarray}
\mathfrak{S}_{\overline{X},\overline{Y},\overline{Z}}\,\{({\nabla}_{\beta\overline{X}}R)(\overline{Y},\overline{Z}, \overline{W})\}&=&0 \label{eq.4}.
\end{eqnarray}
Contracting (\ref{eq.3}) with respect to $\overline{Y}$ and $\overline{Z}$, we get
\begin{equation}\label{aa}
  (\stackrel{h}{\nabla} r)(\overline{X})=r A(\overline{X}).
\end{equation}
From which,
\begin{equation}\label{eq.5}
  (\stackrel{h}{\nabla} r)(\emph{\text{Ric}}_{o}\overline{X})=r A(\emph{\text{Ric}}_{o}\overline{X}).
\end{equation}

On the hand, contracting (\ref{eq.3}) with respect to $\overline{X}$ and $\overline{Y}$ and using (\ref{eq.4}), we obtain
\begin{equation}\label{bb}
   \frac{1}{2}(\stackrel{h}{\nabla} r)(\overline{Z})= A(\emph{\text{Ric}}_{o}\overline{Z}).
\end{equation}
Setting $\overline{Z}=\emph{\text{Ric}}_{o}\overline{X}$, noting that $(M,L)$ is generalized Ricci with associated scalar $\alpha$, (\ref{bb}) becomes
\begin{equation}\label{eq.6}
  \frac{1}{2}(\stackrel{h}{\nabla} r)(\emph{\text{Ric}}_{o}\overline{X})=\alpha A(\emph{\text{Ric}}_{o}\overline{X})
    \end{equation}

Now, (\ref{eq.5}) and (\ref{eq.6}) imply that
\begin{equation}\label{eq.7}
 (r-2 \alpha) A(\emph{\text{Ric}}_{o}\overline{X})=0
\end{equation}
We finally show that $A(\emph{\text{Ric}}_{o}\overline{X})\neq0$. Assume the contrary: $A(\emph{\text{Ric}}_{o}\overline{X})=0$.
From which, taking into account (\ref{aa}) and (\ref{bb}),  we get
$r A=0$. Hence, $r=0$ or $A=0$. Both yield a contradiction. Then, (\ref{eq.7}) implies that $\alpha=\frac{r}{2}$.
\end{proof}

\begin{thm} Let $(M,L)$ be a horizontally integrable Ricci recurrent generalized Ricci Finsler manifold of dimension $n\geq3$ with associated scalar $\alpha$. If $(M,L)$ is Ricci Finsler, then it is three dimensional.
 \end{thm}
\begin{proof}
As $(M,L)$ is horizontally integrable  Ricci recurrent generalized Ricci with associated scalar $\alpha$. Then, from
Theorem \ref{thm.2}\textbf{(d)}, we have
\begin{equation}\label{eq.8}
  \alpha=\frac{r}{2}.
\end{equation}
On the other hand, if $(M,L)$ is Ricci Finsler, then the associated scalar $\alpha$ has the form
\begin{equation}\label{eq.9}
 \alpha=\frac{r}{n-1}
\end{equation}
As $(M,L)$ is horizontally integrable, the Ricci tensor is symmetric. Consequently, by Theorem \ref{thm.2}\textbf{(a)},
the proof follows immediately from (\ref{eq.8}) and (\ref{eq.9}).
\end{proof}

\begin{thm} \label{thm.3} Every Finsler manifold of dimension $n\geq3$ satisfying $\emph{\text{Ric}}=\frac{r}{n} \, g$ is a generalized Ricci Finsler manifold with associated scalar $\alpha=\frac{r}{n}$.
 \end{thm}
\begin{proof}
The proof is clear and we omit it.
\end{proof}

\begin{defn} Let $(M,L)$ be a Finsler manifold  of dimension $n\geq3$ with non-zero $h$-curvature tensor
 ${R}$. We will say that $(M,L)$ is a semi-isotropic Finsler manifold if the $h$-curvature $R$ has the form:
 $$R(\overline{X},\overline{Y},\overline{Z},\overline{W})=A(\overline{X},\overline{Z})A(\overline{Y},\overline{W})
 -A(\overline{X},\overline{W})A(\overline{Y},\overline{Z}),$$
where $A$ is a non-zero symmetric tensor of type (0,2), called the associated tensor.
\end{defn}

\begin{thm}\label{thm.4} Every horizontally integrable  semi-isotropic Finsler manifold, with associated tensor as the Ricci tensor, is  generalized Ricci with associated scalar $\alpha=r-1$.
 \end{thm}
\begin{proof}
As the Ricci  tensor of a horizontally integrable Finsler manifold is symmetric \cite{Types},  then, we have
\begin{equation*}\label{eq.10}
  R(\overline{X},\overline{Y},\overline{Z},\overline{W})=\emph{\text{Ric}}(\overline{X},\overline{Z})\emph{\text{Ric}}(\overline{Y},\overline{W})
 -\emph{\text{Ric}}(\overline{X},\overline{W})\emph{\text{Ric}}(\overline{Y},\overline{Z}).
\end{equation*}
Contracting both sides of the above equation with respect to $\overline{Y}$ and $\overline{W}$, we obtain
$$\emph{\text{Ric}}(\overline{X},\overline{Z})=r\emph{\text{Ric}}(\overline{X},\overline{Z})-
g(\emph{\text{Ric}}_{o}\overline{X},\emph{\text{Ric}}_{o}\overline{Z}).$$
From which, noting that the Ricci  tensor $\text{Ric}$ is symmetric
$$\emph{\text{Ric}}(\emph{\text{Ric}}_{o}\overline{X},\overline{Z})=(r-1)\emph{\text{Ric}}(\overline{X},\overline{Z}).$$
Hence, $(M,L)$ is generalized Ricci with associated scalar $\alpha=r-1$.
\end{proof}

\begin{rem}
Theorem \ref{thm.3} and Theorem \ref{thm.4} give two classes of generalized Ricci Finsler manifolds.
\end{rem}

\Section{Projective (m-projective) recurrence}

In this section, we investigate some new types of recurrent Finsler spaces, namely the projectively recurrent and m-projectively recurrent Finsler spaces.
Some Finsler tensors  are defined and their properties are studied.  These tensors are used to define the projectively (m-projectively) recurrent Finsler space.

\smallskip
For  a Finsler manifold  of dimension $n\geq3$ with non-zero Ricci tensor
 $\textmd{Ric}$. we define the following tensors:
 \begin{eqnarray}
   \mathbb{P}(\overline{X},\overline{Y})\overline{Z} &:=& R(\overline{X},\overline{Y})\overline{Z}-\frac{1}{(n-1)}\{
   \emph{\text{Ric}}(\overline{X},\overline{Z})\overline{Y}-\emph{\text{Ric}}(\overline{Y},\overline{Z})\overline{X}\} , \label{eq.b1}\\
   \mathbb{H}(\overline{X},\overline{Y})\overline{Z} &:=& R(\overline{X},\overline{Y})\overline{Z}-\frac{1}{2(n-1)}\{
   \emph{\text{Ric}}(\overline{X},\overline{Z})\overline{Y}-\emph{\text{Ric}}(\overline{Y},\overline{Z})\overline{X} \nonumber\\
  {\qquad\qquad\quad}  &&  +g(\overline{X},\overline{Z})\emph{\text{Ric}}_{o}\overline{Y}
  -g(\overline{Y},\overline{Z})\emph{\text{Ric}}_{o}\overline{X}\}. \label{eq.b2}
 \end{eqnarray}
The tensor $\mathbb{P}$ (resp. $\mathbb{H}$) is called the projective (resp. m-projective) curvature tensor.

\smallskip

\begin{defn}\label{def.2a} Let $(M,L)$ be a Finsler manifold  of dimension $n\geq3$ with non-zero  Ricci tensor.
 Then, $(M,L)$ is said to be:
 \begin{description}
   \item[(a)]  projectively recurrent if \,$\stackrel{h}\nabla \mathbb{P}=A\otimes \mathbb{P},$

\item[(b)] m-projectively recurrent if \,$\stackrel{h}\nabla \mathbb{H}=A\otimes \mathbb{H},$
  \end{description}
       where $A$ is a non-zero $\pi$-form on $TM$ called the associated form.
   \par 
In particular, if \,$\stackrel{h}\nabla \mathbb{P}=0$ \emph{(}resp. $\stackrel{h}\nabla \mathbb{H}=0$\emph{)} , then $(M,L)$ is
called projectively \emph{(}resp. m-projectively\emph{)} symmetric.

\end{defn}

The following result gives some properties for the m-projective curvature tensor.
\begin{prop}\label{thm.5}  Let $(M,L)$ be a Finsler manifold  of dimension $n\geq3$ with non-zero Ricci  tensor.
 Then, the m-projective curvature  tensor $\mathbb{H}$  has the following properties\,:\vspace{-0.2cm}
\begin{description}
 \item[(a)] $ \mathbb{H}(\overline{X},\overline{Y},\overline{Z}, \overline{W})=
 - \mathbb{H}(\overline{Y},\overline{X},\overline{Z},\overline{W})$,

 \item[(b)]  $ \mathbb{H}(\overline{X},\overline{Y},\overline{Z}, \overline{W})=
 - \mathbb{H}(\overline{X},\overline{Y},\overline{W},\overline{Z})$,

 \item[(c)]  $\mathfrak{S}_{\overline{X},\overline{Y},\overline{Z}}\,
\{\mathbb{H}(\overline{X},
\overline{Y})\overline{Z}\}=\mathfrak{S}_{\overline{X},\overline{Y},\overline{Z}}\,\{T(\widehat{R}(\overline{X},\overline{Y}),\overline{Z})
-\frac{1}{2(n-1)}[ \emph{\text{Ric}}(\overline{X},\overline{Z})\overline{Y}-\emph{\text{Ric}}(\overline{Y},\overline{Z})\overline{X}
   +g(\overline{X},\overline{Z})\emph{\text{Ric}}_{o}\overline{Y}
  -g(\overline{Y},\overline{Z})\emph{\text{Ric}}_{o}\overline{X}]\}$,

\item[(d)] $\mathfrak{S}_{\overline{X},\overline{Y},\overline{Z}}\,
\{(\nabla_{\beta \overline{X}}\mathbb{H})(\overline{Y},
\overline{Z},\overline{W})\}=-\mathfrak{S}_{\overline{X},\overline{Y},\overline{Z}}\,
\{P(\overline{X},\widehat{R}(\overline{Y},\overline{Z}))\overline{W}+\frac{1}{2(n-1)}[
   (\nabla_{\beta \overline{X}}\emph{\text{Ric}})(\overline{Y},\overline{W})\overline{Z}-(\nabla_{\beta \overline{X}}\emph{\text{Ric}})(\overline{Z},\overline{W})\overline{Y}
   +g(\overline{Y},\overline{W})(\nabla_{\beta \overline{X}}\emph{\text{Ric}}_{o})(\overline{Z})
  -g(\overline{Z},\overline{W})(\nabla_{\beta \overline{X}}\emph{\text{Ric}}_{o})(\overline{Y})]\}$,

\item[(e)] $(\nabla_{\gamma \overline{\eta}}\mathbb{H})(\overline{X}, \overline{Y},
\overline{Z})=0$,

\item[(f)] $\mathbb{H}(\overline{X}, \overline{Y})\overline{\zeta}=\frac{1}{2(n-1)}\{g(\overline{Y},\overline{\zeta})\emph{\text{Ric}}_{o}\overline{X}
    -g(\overline{X},\overline{\zeta})\emph{\text{Ric}}_{o}\overline{Y}\}$,
\end{description}
where $\overline{\zeta}$ is a concurrent $\pi$-vector field \cite{r94a}.
\end{prop}

\begin{proof} The proof follows from Theorem 3.6 of \cite{r96} and Proposition 2.4 of \cite{r94a} together with the definition of m-projective curvature tensor.
 \end{proof}

 \begin{thm}\label{thm.5} For a Finsler manifold with non-zero Ricci tensor
 ${\textmd{Ric}}$ satisfying \linebreak $\emph{\text{Ric}}=\frac{r}{n} \, g$,  the three  notions of being  concircularly recurrent,
   projectively  recurrent and m-projectively recurrent  are equivalent.
\end{thm}
\begin{proof} The proof follows from the fact that the concircular curvature tensor $C$, the projective curvature tensor $\mathbb{P}$
and the m-projective curvature tensor $\mathbb{H}$  coincide under the given assumption $\emph{\text{Ric}}=\frac{r}{n} \, g$.
 \end{proof}
We know that every recurrent Finsler manifold is Ricci recurrent (Theorem 3.2\textbf{(a)} of \cite{Types}). The converse of this theorem is not true.
For the converse to be true we need an additional assumption as shown in the next result.
\begin{thm} A Ricci recurrent m-projectively recurrent  Finsler manifold with the same recurrence form is recurrent.
\end{thm}
\begin{proof} As $(M,L)$ is Ricci recurrent with recurrence form $A$, then, we have \cite{Types}
\begin{equation}\label{eq.b3}
  (\nabla_{\beta \overline{W}} \emph{\text{Ric}})(\overline{X},\overline{Y})= A(\overline{W}) \emph{\text{Ric}}(\overline{X},\overline{Y}).
\end{equation}
Applying the h-covariant derivative on both sides of (\ref{eq.b2}), noting that $\nabla g=0$, we get
 \begin{eqnarray*}
      (\nabla_{\beta \overline{W}}\mathbb{H})(\overline{X},\overline{Y},\overline{Z}) &=& (\nabla_{\beta \overline{W}}R)(\overline{X},\overline{Y},\overline{Z})-\frac{1}{2(n-1)}\{
   (\nabla_{\beta \overline{W}}\emph{\text{Ric}})(\overline{X},\overline{Z})\overline{Y}   \nonumber\\
    && -(\nabla_{\beta \overline{W}}\emph{\text{Ric}})(\overline{Y},\overline{Z})\overline{X}  +g(\overline{X},\overline{Z})(\nabla_{\beta \overline{W}}\emph{\text{Ric}}_{o})(\overline{Y})\nonumber\\
    && -g(\overline{Y},\overline{Z})(\nabla_{\beta \overline{W}}\emph{\text{Ric}}_{o})(\overline{X})\},
 \end{eqnarray*}
In view of (\ref{eq.b3}), this equation becomes
\begin{eqnarray}
      (\nabla_{\beta \overline{W}}\mathbb{H})(\overline{X},\overline{Y},\overline{Z}) &=& (\nabla_{\beta \overline{W}}R)(\overline{X},\overline{Y},\overline{Z})-\frac{A(\overline{W})}{2(n-1)}\{
   \emph{\text{Ric}}(\overline{X},\overline{Z})\overline{Y}-\emph{\text{Ric}}(\overline{Y},\overline{Z})\overline{X} \nonumber\\
  {\qquad\qquad\quad}  &&  +g(\overline{X},\overline{Z})\emph{\text{Ric}}_{o}\overline{Y}
  -g(\overline{Y},\overline{Z})\emph{\text{Ric}}_{o}\overline{X}\}. \label{eq.b4}
 \end{eqnarray}

Now, let  $(M,L)$ be an m-projectively recurrent manifold with the same recurrence form $A$.
 Then, from Definition \ref{def.2a} and  (\ref{eq.b4}), we obtain
$$(\nabla_{\beta \overline{W}}R)(\overline{X},\overline{Y},\overline{Z}) = A(\overline{W})R(\overline{X},\overline{Y})\overline{Z}.$$
Hence, $(M,L)$ is recurrent with the same recurrence form $A$.
\end{proof}

\begin{thm} Each recurrent Finsler manifold is  m-projectively recurrent.
\end{thm}
\begin{proof}
Since  $(M,L)$ is recurrent with  recurrence form $A$,
then $(M,L)$ is Ricci recurrent with the same recirrence form $A$.
Using  Definition 2.1\textbf{(a)} of \cite{Types}, taking into account (\ref{eq.b4}) and  (\ref{eq.b2}), we conclude that
$$(\nabla_{\beta \overline{W}}\mathbb{H})(\overline{X},\overline{Y},\overline{Z}) = A(\overline{W})\mathbb{H}(\overline{X},\overline{Y})\overline{Z}.$$
Hence, $(M,L)$ is  m-projectively recurrent with the same recurrence form $A$.
 \end{proof}

\begin{thm} Let $(M,L)$ be a Ricci recurrent Finsler manifold. Then, $(M,L)$
 is m-projectively recurrent if and only if it is projectively recurrent with the same recurrence form.
\end{thm}
\begin{proof} From (\ref{eq.b1}) and (\ref{eq.b2}), we have
\begin{eqnarray}
  \mathbb{H}(\overline{X},\overline{Y})\overline{Z} &:=& \mathbb{P}(\overline{X},\overline{Y})\overline{Z}+\frac{1}{2(n-1)}\{
   \emph{\text{Ric}}(\overline{X},\overline{Z})\overline{Y}-\emph{\text{Ric}}(\overline{Y},\overline{Z})\overline{X} \nonumber\\
    &&  +g(\overline{Y},\overline{Z})\emph{\text{Ric}}_{o}\overline{X}
  -g(\overline{X},\overline{Z})\emph{\text{Ric}}_{o}\overline{Y}\}, \label{eq.bb2}
\end{eqnarray}
From which, taking the h-covariant derivative of both sides, we obtain
 \begin{eqnarray*}
      (\nabla_{\beta \overline{W}}\mathbb{H})(\overline{X},\overline{Y},\overline{Z}) &=& (\nabla_{\beta \overline{W}}\mathbb{P})(\overline{X},\overline{Y},\overline{Z})+\frac{1}{2(n-1)}\{
   (\nabla_{\beta \overline{W}}\emph{\text{Ric}})(\overline{X},\overline{Z})\overline{Y}   \nonumber\\
    && -(\nabla_{\beta \overline{W}}\emph{\text{Ric}})(\overline{Y},\overline{Z})\overline{X}  +g(\overline{Y},\overline{Z})(\nabla_{\beta \overline{W}}\emph{\text{Ric}}_{o})(\overline{X})\nonumber\\
    && -g(\overline{X},\overline{Z})(\nabla_{\beta \overline{W}}\emph{\text{Ric}}_{o})(\overline{Y})\},
 \end{eqnarray*}
Since, $(M,L)$ is Ricci recurrent with recurrence form $A$, then by (\ref{eq.b3}), the above equation takes the form
\begin{eqnarray}
      (\nabla_{\beta \overline{W}}\mathbb{H})(\overline{X},\overline{Y},\overline{Z}) &=& (\nabla_{\beta \overline{W}}\mathbb{P})(\overline{X},\overline{Y},\overline{Z})+\frac{A(\overline{W})}{2(n-1)}\{
   \emph{\text{Ric}}(\overline{X},\overline{Z})\overline{Y}-\emph{\text{Ric}}(\overline{Y},\overline{Z})\overline{X} \nonumber\\
  {\qquad\qquad\quad}  &&  +g(\overline{Y},\overline{Z})\emph{\text{Ric}}_{o}\overline{X}
  -g(\overline{X},\overline{Z})\emph{\text{Ric}}_{o}\overline{Y}\}, \label{eq.bb4}
 \end{eqnarray}

Now, let  $(M,L)$ be m-projectively recurrent with the same recurrence form $A$. Then, from Definition \ref{def.2a},
 taking into account (\ref{eq.bb2}), the above equation reduces to
$$(\nabla_{\beta \overline{W}}\mathbb{P})(\overline{X},\overline{Y},\overline{Z}) = A(\overline{W})\mathbb{P}(\overline{X},\overline{Y})\overline{Z}.$$
Hence, $(M,L)$ is projectively recurrent with the same recurrence form $A$.
\par
Conversely, let  $(M,L)$ be projectively recurrent with the same recurrence form $A$. Then, from Definition \ref{def.2a}, taking into account  (\ref{eq.bb4}) and  (\ref{eq.bb2}), we conclude that $(M,L)$ is  m-projectively recurrent with the same recurrence form $A$.
 \end{proof}


\providecommand{\bysame}{\leavevmode\hbox
to3em{\hrulefill}\thinspace}
\providecommand{\MR}{\relax\ifhmode\unskip\space\fi MR }
\providecommand{\MRhref}[2]{%
  \href{http://www.ams.org/mathscinet-getitem?mr=#1}{#2}
} \providecommand{\href}[2]{#2}

\end{document}